\newcommand{\R}{\mathds R}

\newcommand{\N}{\mathds N}
\newcommand{\dist}{\mathrm{dist}}

\newcommand{\Ddt}{\tfrac{\mathrm D}{\mathrm dt}}
\newcommand{\Dds}{\tfrac{\mathrm D}{\mathrm ds}}

\documentclass[twoside,a4paper,11pt]{amsart}

\usepackage{amsmath}
\usepackage{times}

\usepackage{hyperref}
\usepackage{dsfont}
\usepackage{graphicx}

\usepackage{amssymb}

\numberwithin{equation}{section}

\title[Morse Theory for geodesics in degenerate metrics]{Morse Theory for geodesics\\ in singular conformal metrics}
\author[R.\ Giamb\`o ,\ F.\ Giannoni]{Roberto Giamb\`o, Fabio Giannoni}
\address{Scuola di Scienze e Tecnologie \hfill\break\indent
Universit\`a di Camerino\hfill\break\indent Italy}
\email{roberto.giambo@unicam.it, fabio.giannoni@unicam.it}
\author[P. Piccione]{Paolo Piccione}
\address{Departamento de Matem\'atica\hfill\break\indent Instituto de
Matem\'atica e Estat\'\i stica
\hfill\break\indent Universidade de S\~ao Paulo
\hfill\break\indent Brazil}
\email{piccione@ime.usp.br}


\date{December 23rd, 2013}

\begin{document}


\theoremstyle{plain}\newtheorem{teo}{Theorem}[section]
\theoremstyle{plain}\newtheorem{prop}[teo]{Proposition}
\theoremstyle{plain}\newtheorem{lem}[teo]{Lemma}
\theoremstyle{plain}\newtheorem{cor}[teo]{Corollary}
\theoremstyle{definition}\newtheorem{defin}[teo]{Definition}
\theoremstyle{remark}\newtheorem{rem}[teo]{Remark}
\theoremstyle{definition}\newtheorem{example}[teo]{Example}
\theoremstyle{remark}\newtheorem{step}{\bf Step}
\theoremstyle{plain}\newtheorem*{teon}{Theorem}
\theoremstyle{plain}\newtheorem*{conj}{Conjecture}
\theoremstyle{plain}\newtheorem*{defin*}{Definition}


\begin{abstract}
Motivated by the use of degenerate Jacobi metrics for the study of brake orbits and homoclinics, as done in \cite{GGP1,esistenza, arma}, we develop a Morse theory
for geodesics in conformal metrics having conformal factors vanishing on a regular hypersurface of a Riemannian manifold.
\end{abstract}

\maketitle
\renewcommand{\contentsline}[4]{\csname nuova#1\endcsname{#2}{#3}{#4}}
\newcommand{\nuovasection}[3]{\medskip\hbox to \hsize{\vbox{\advance\hsize by -1cm\baselineskip=12pt\parfillskip=0pt\leftskip=3.5cm\noindent\hskip -2cm #1\leaders\hbox{.}\hfil\hfil\par}$\,$#2\hfil}}
\newcommand{\nuovasubsection}[3]{\medskip\hbox to \hsize{\vbox{\advance\hsize by -1cm\baselineskip=12pt\parfillskip=0pt\leftskip=4cm\noindent\hskip -2cm #1\leaders\hbox{.}\hfil\hfil\par}$\,$#2\hfil}}

\section{Introduction}\label{sec:intro}
Let $N\ge2$, denote by $(q,p)\!=\!(q_1,\ldots,q_N,p_1,\ldots,p_N)$ the canonical coordinates
in $\R^{2N}$, and consider a $C^2$ Hamiltonian function $H:\R^{2N}\to \R$ of the form
\begin{equation}\label{eq:Ham}
H(q,p)=\frac12\sum_{i,j=1}^N a^{ij}(q)p_ip_j+V(q),
\end{equation}
where $q\mapsto(a^{ij}(q))_{i,j}$ is a map of class $C^2$ taking value in
the space of symmetric positive definite $N\times N$ matrices, and $V:\R^N\to\R$
is the potential energy.
The corresponding Hamiltonian system is:
\begin{equation}\label{eq:HS}
\dot p=-\frac{\partial H}{\partial q},\quad
\dot q=\frac{\partial H}{\partial p},
\end{equation}
where the dot denotes differentiation with respect to time. Since
the system \eqref{eq:HS} is  time independent, the function $H$
is constant along each solution; this value is called the \emph{total energy} of the given solution.
There exists a huge amount of literature concerning the study of periodic solutions of autonomous
Hamiltonian systems with prescribed energy (see e.g., \cite{LZZ,Long,rab}
and references therein).

A special class of periodic solutions of \eqref{eq:HS} are the so
called \emph{brake orbits}. A brake orbit for the system \eqref{eq:HS} is a
nonconstant solution $(q,p):\R\to\R^{2N}$ such that $p(0)=p(T)=0$
for some $T>0$; since $H$ is even in the momenta $p$, then a brake orbit is $2T$--periodic.
Moreover, if $E$ is the energy of a brake orbit $(q,p)$, then $V(q(0))=V(q(T))=E$.
Obviously, such a notion can be given when the quadratic form $\sum_{i,j=1}^N a_{ij}(x)\mathrm dx_i\mathrm dx_j$
is replaced by any $C^2$--Riemannian metric $g$ on a manifold $M$.

In \cite{GGP1} it is shown that the study of brake orbits can be reduced to the study of Orthogonal Geodesic Chords in suitable manifolds with boundary. This was done using the Jacobi metric and the related distance function from the boundary of the potential well, for small values of the distance function. (This kind of results was applied in \cite{arma} to obtain multiplicity results for brake orbits and homoclinics also). 

The use of the Jacobi metric and the related distance function from the boundary was introduced by Seifert in \cite{seifert}
to prove the existence of at least one brake orbit, assuming the potential well ${V^{-1}(]-\infty,E]})$ to be homeomorphic to the $N$--dimensional disk and to have smooth boundary. In \cite{seifert} Seifert gave also some asymptotic estimates of brake orbits nearby $V^{-1}(E)$ which are a crucial point for our analysis of the Jacobi fields along geodesics with respect to the Jacobi metric and starting from $V^{-1}(E)$.

In this paper we will study the distance function from the boundary of the potential well also for large values of the distance, using again the classical Maupertuis principle (see~Proposition \ref{prop:MP}), which states that the solutions of \eqref{eq:HS} for the natural Hamiltonian \eqref{eq:Ham}, having total energy $E$, are reparameterizations of geodesics relatively to the Jacobi metric:
\begin{equation}\label{eq:Ejacobi}
g^E=\tfrac12\big(E-V(x)\big)\sum_{i,j=1}^N a_{ij}(x)\,\mathrm dx_i\,\mathrm dx_j,
\end{equation}
where $(a_{ij})$ denotes the inverse matrix of $(a^{ij})$.

The purpose of the present paper is to study the differentiability and related properties of the function distance from the boundary of the potential well.
We will prove that, also for this degenerate case, properties similar to the nonsingular case are still satisfied. More generally, we will consider a
Riemannian manifold $(M,g)$ of class $C^3$, with dimension $N\geq 1$, a map $V:M\to \mathds R$, and the conformal metric:
\begin{equation}\label{eq:Econf}
g_*=\tfrac12(E-V)g,
\end{equation}
$E \in \mathds{R}$, defined in the sublevel $V^{-1}\left(\left]-\infty,E\right[\right)$; this sublevel will called the \emph{potential well}.
We shall denote by $\text{dist}_*$ the distance induced by $g_*$ and by $\text{dist}$ the distance induced by $g$. We will
make the following assumptions:
\begin{itemize}\label{eq:Jacobi}
\item $V$ is of class $C^2$ in a neighborhood of $V^{-1}\left(\left]-\infty,E\right[\right)$;
\item $E$ is a regular value for $V$;
\item the sublevel $V^{-1}\left(\left]-\infty,E\right]\right)$ is compact.
\end{itemize}
To avoid further unessential techicalities
we shall assume that $M$ is topologocally embedded as an open subset of $\R^N$ for some $N$.
Under these assumptions we will prove some regularity properties of $\text{dist}_*$. In Proposition~\ref{prop:prop3.6}, under the assumption of
uniqueness of the minimizer, we improve the result \cite[Proposition~5.6]{GGP1}. Let us recall that \cite[Proposition~5.6]{GGP1} establishes
the regularity of the distance function only near the singular boundary of the potential well, where the uniqueness of the minimizer is guaranteed.

Our study will naturally lead  to a formulation of the Morse index theorem for $g_*$-geodesics in the manifold with singular boundary
$V^{-1}\left(\left]-\infty,E\right]\right)$.
We will prove that, for the type of singularity considered here,
one obtains a theory analogous to the fixed endpoint theory in classical Riemannian
geometry. More precisely, recall that the classical Morse theory for orthogonal geodesics
involves the notion of focal point, which is determined by the curvature tensor of the metric,
and the extrinsic geometry of the initial submanifold, i.e., its second fundamental form.
We will show here that, when the metric is assumed to degenerate on the initial submanifold
in the appropriate way, then the contribution of the geometry of the initial manifold disappears,
as if it had collapsed to a single point. For the precise statement of our results, see Proposition~\ref{valutazione-hessiano} and
Theorem~\ref{MIT}. Clearly, such result is relevant in the context of infinite dimensional Morse theory for
geodesics in manifolds with singular boundary, that can be employed to give lower estimates on the number
of periodic orbits of Hamiltonian systems.
\medskip

The paper is organized as follows. In section
\ref{sec:MP} we recall Maupertuis principle stated with curves of class $H^1$. In section \ref{minimal-geodesics} we study the differentiability of the distance
function from the boundary of the potential well. In section
\ref{sec:nondegen} we study conjugate points and Jacobi fields for geodesics joining
$V^{-1}(E)$ with a point in $V^{-1}(]-\infty,E[)$, proving the Morse Index Theorem.
In a forthcoming paper we shall describe  the exponential map and its principle properties with applications to the $C^2$-regularity of the distance from the boundary of the potential well.

\section{Maupertuis Principle}\label{sec:MP}
Maupertuis principle will allow us to obtain suitable estimates for geodesics with respect to the Jacobi metric \eqref{eq:Ejacobi}.
Denote by $H^1([a,b],\R^N)$
the Sobolev space of the absolute continuous curves from $[a,b]$ to $\R^N$ having derivative in $L^2([a,b])$
and consider the Maupertuis integral
$f_{a,b}:H^1\big([a,b],\mathds R^N)\to\R$, which is the geodesic
action functional relative to the metric
\eqref{eq:Ejacobi}, given by:
\begin{equation}\label{eq:maupint}
f_{a,b}(x)=\int_a^b\frac12\big(E-V(x)\big)g\big(\dot x,\dot
x\big)\, \mathrm ds.
\end{equation}

The functional $f_{a,b}$ is smooth, and its differential is
readily computed as:
\begin{equation}\label{eq:diffmaupint}
\mathrm df_{a,b}(x)\xi=\int_a^b\big(E-V(x)\big)g\big(\dot x,\Dds
\xi\big)\, \mathrm ds-\frac12\int_a^bg\big(\dot x,\dot x\big)
g(\nabla V(x),\xi\big)\,\mathrm ds,
\end{equation}
where $\xi\in H^1\big([a,b],\R^N\big)$, $\Dds$ denote the covariant derivative along the curve $x$ and $\nabla$ is the gradient with respect to the metric $g$.
The corresponding
Euler--Lagrange equation of the critical points of $f_{a,b}$ is
\begin{multline}\label{eq:pr1}
\big(E-V(x(s))\big)\Dds\dot x(s) - g\big(\nabla V(x(s)),\dot
x(s)\big)\dot x(s) +\\
\frac 12 g\big(\dot x(s),\dot x(s)\big)\nabla
V(x(s))=0,
\end{multline}
for all $s\in ]a,b[$. A solution of \eqref{eq:pr1} of \eqref{eq:pr1} will be called $g_*$--geodesic.

Solutions of the Hamiltonian system \eqref{eq:HS} having fixed
energy $E$ and critical points of the functional $f_{a,b}$ of
\eqref{eq:maupint} are related by the following variational
principle, known in the literature as the {\em Maupertuis--Jacobi
principle}.
\begin{prop}\label{prop:MP}
\label{thm:MauJac} Assume that the potential well $ V^{-1}\big(\left]-\infty,E\right[\big)\ne\emptyset$, where $E$ is a regular value of the
function $V$.

Let $x\in C^0\big([a,b],\R^N\big)\cap
H^1_{\textrm{loc}}\big(\left]a,b\right[,\R^N\big)$ be a non
constant curve such that
\begin{equation}\label{eq:i}
\int_a^b\big(E-V(x)\big)g\big(\dot x,\Dds \xi\big)\, \mathrm
ds-\frac12\int_a^bg\big(\dot x,\dot x\big) g(\nabla
V(x),\xi\big)\,\mathrm ds=0
\end{equation}
for all $\xi\in C^\infty_0\big(\left]a,b\right[,\R^N\big)$, and such
that:
\begin{equation}\label{thm:eq:ii}
V\big(x(s)\big)<E,\quad\text{for all $s\in\left]a,b\right[$;}
\end{equation}
and
\begin{equation}\label{eq:iii}
V\big(x(a)\big),V\big(x(b)\big)\le E.
\end{equation}
Then, $x\in H^1\big([a,b],\R^N\big) \cap C^2(]a,b[)$, and if
$\,\,V\big(x(a)\big)=V\big(x(b)\big)=E$, it is $x(a)\ne x(b)$.
Moreover, in the above situation, there exist  positive constants
$c_x$ and $T$, and a $C^1$-diffeo\-morphism $\sigma:[0,T]\to[a,b]$
such that:
\begin{equation}\label{eq:3p2}
\frac12\big(E-V(x)\big)g\big(\dot x,\dot x\big)\equiv c_x\quad \text{on
$]a,b[$},
\end{equation}
and, setting $q=x\circ\sigma:[0,T]\to\R^N$,
$q$ satisfies
\begin{equation}\label{eq:eqbo}
\Dds \dot q + \nabla V(q) = 0, \; \tfrac12g(\dot q,\dot q) + V(q) = E, \text{ for all }t \in [0,T].
\end{equation}
Furthermore, $q(0)=x(a), q(T)=x(b)$, and if  $V\big(x(a)\big)=V\big(x(b)\big)=E$ then $q$ can be extended to
a $2T$-periodic brake orbit.
\end{prop}
Note that the existence of a constant $c_x$ for whicg \eqref{eq:3p2} is satisfied is obtained readily from \eqref{eq:pr1},
contracting both sides of the equality  with $\dot x(s)$.

To prove Proposition \ref{prop:MP} we need the following results that shall be used also for the study of the Morse index.

\begin{lem}\label{orthcond} There exists a positive constant $C$ such that for any
$C^2$--solution $q$ of
\begin{equation}\label{eq:newteq}
\frac{D}{dt}\dot q + \nabla V(q)=0
\end{equation}
in any interval $[0,t_0]$ with $q(0)\in V^{-1}(E)$, $\dot q(0)=0$ and $q(t) \in V^{-1}(]-\infty,E[)$ for any $t \in ]0,t_0]$, we have
\begin{equation}\label{eq:uniform-est}
\left\Vert \frac{\nabla V(q(t))}{\Vert \nabla V(q(t)) \Vert} + \frac{\dot q(t)}{\Vert \dot q(t)\Vert}\right\Vert \leq Ct, \text{ for any }t \in [0,t_0]
\end{equation}
where $\Vert \cdot \Vert$ denotes the norm induced by the Riemannian metric $g$.
\end{lem}

\begin{proof}
Since $\dot q$ is of class $C^1$, $\dot q(0)=0$ and $V^{-1}(]-\infty,E])$ is bounded  there exists a continuous vector field $\Lambda$, which is
bounded by constants independent of $q$,
such  that,
\[
\dot q(t) = -t\nabla V(q(0)) + t^2\Lambda(t).
\]
as we see using equation \eqref{eq:newteq} and the first order Taylor expansion of $\nabla V$ at $q(0)$.

Then there exists a continuous map $\Delta$ (bounded independently of $q$) such that
\begin{multline*}
\left\Vert \frac{\nabla V(q(t))}{\Vert \nabla V(q(t)) \Vert} + \frac{\dot q(t)}{\Vert \dot q(t)\Vert}\right \Vert =
\frac{1}{\Vert -\nabla V(q(0))t +  \Lambda(t)t^2\Vert}\frac{1}{\Vert \nabla V(q(t)) \Vert} \\
\cdot \Big( \Big\Vert \nabla V(q(t)) \Vert \nabla V(q(0))\Vert t - \nabla V(q(0))t \Vert \nabla V(q(t))\Vert \Big \Vert + \Delta(t)t^2\Big).
\end{multline*}
Since $\nabla V(q(0)) \not=0$ the thesis follows considering the first order  Taylor expansion of the vector field
\[\nabla V(q(t)) \Vert \nabla V(q(0))\Vert  - \nabla V(q(0)) \Vert \nabla V(q(t)) \Vert \]
which  is infinitesimal (as $t$ goes to $0$) and of class $C^1$.
\end{proof}

\begin{lem}\label{lem:map-class}
Let $x: [a,b] \rightarrow V^{-1}(]-\infty,E[)$ a non constant $C^2$--solution of the differential equation \eqref{eq:pr1}. Let $c_x$ a positive real constant
such that \eqref{eq:3p2} is satisfied. Consider the map
\begin{equation}\label{eq:tempo}
t(s)=\int_{a}^{s}\frac{\sqrt{c_x}}{E-V(x(\tau))}\,\text d\tau,
\end{equation}
denote by $\sigma(t)$ the inverse of $t(s)$ and consider $q(t) = x(\sigma(t))$. Then $q$ satisfies
\eqref{eq:eqbo}.
\end{lem}
\begin{proof}
It is a strightforward calculation, using the fact that
$\sigma'(t)=\frac{E-V(x(\sigma(t))}{\sqrt{c_x}}$.
\end{proof}

\begin{rem}\label{E-Vat-the-boundary}
Let $x$ be a non constant $C^2$--solution of the differential equation \eqref{eq:pr1};
it satisfies \eqref{eq:3p2} with $c_x$ positive real constant.
Note that
\begin{multline}\label{eq:ODEE-V}
\frac{d}{ds}(E-V(x(s)) = -g(\nabla V(x(\tau)),\dot x(\tau))=\\
-\frac{g(\nabla V(x(\tau)),\dot x(\tau))}{\Vert\nabla V(x(\tau)\Vert\Vert\dot x(\tau)\Vert}\Vert\nabla V(x(\tau)\Vert
\Vert \dot x(\tau)\Vert,
\end{multline}
Since $\nabla V \not= 0$ near $V^{-1}(E)$ and $\Vert \dot x(\tau)\Vert = \frac{\sqrt{2c_x}}{\sqrt{E-V(x(\tau))}}$, using Lemma \ref{orthcond} and classical
comparison theorems for ordinary differential equations, we see that the behavior of $E-V(x(s))$ near the boundary of the potential well,
is the same of the solutions of the differential equation $\dot y = \frac{1}{\sqrt{y}}$ nearby $y=0$.

In particular if $x(s)$ reaches the boundary at some instant $s_0$, the map $E-V(x(s))$ behaves near $s_0$ as $(s-s_0)^{\frac23}$.
This shows that
the map \eqref{eq:tempo} is bounded. Moreover, thanks to the uniform estimates \eqref{eq:uniform-est} in Lemma~\ref{orthcond}, we see that,
when $c_x$ is bounded independently
of $x$, the map \eqref{eq:tempo} is uniformly bounded (independently on $x$).
\end{rem}

\begin{proof}[Proof of Proposition \ref{prop:MP}]

Since $x$ satisfies \eqref{eq:i}, standard regularization
arguments show that $x$ is of class $C^2$ on $]a,b[$.
Integration by parts gives \eqref{eq:pr1}, for all $s\in]a,b[$.
Equation \eqref{eq:3p2} follows contracting both sides of
\eqref{eq:pr1} with $\dot x$ using $g$. Now, define $t(s)$ as in \eqref{eq:tempo}.
By Remark \ref{E-Vat-the-boundary}, the real map $t(s)$ in \eqref{eq:tempo} is well defined for all $s \in [a,b]$ and $T\equiv t(b)<+\infty$.
Denoting by $\sigma(t)$ the inverse of $t(s)$, by Lemma \ref{lem:map-class} we deduce that the curve $q(t)= x(\sigma(t))$
satisfies
\eqref{eq:eqbo}.
Moreover $q(0)=x(a)$ and $q(T)=x(b)$, and by the uniqueness of the solution of the
Cauchy problem, if $V(x(a))=V(x(b))=E$ it must be $q(0)\not=q(T)$,
and $q$ can be extended to a periodic $2T$--periodic solution of \eqref{eq:eqbo}, namely a brake orbit.
\end{proof}

\section{Minimal geodesics}\label{minimal-geodesics}
Let
$\Omega_E= V^{-1}(]-\infty,E[)$ and recall that  $\nabla
V(x)\not=0$ for all $x\in V^{-1}(E)$, and that $\overline{\Omega}_E$ is
compact.
For any $Q \in \Omega_E$ set
\begin{equation}\label{eq:Xspace}
X_Q=\big\{x\in H^1([0,1],\overline\R^N)\,:\,x(0) \in \partial \Omega_E,
x(]0,1])\subset\Omega_E, x(1)=Q\big\}.
\end{equation}

\begin{lem}\label{thm:lem8.2}
For all $Q\in\Omega_E$, the infimum:
\begin{equation}\label{eq:dV}
d_V(Q):=\inf\left\{\int_0^1\sqrt{\tfrac12\left((E-V(x))g\big(\dot x,\dot x\big)
\right)}\,\mathrm ds\;:\;x\in X_Q\right\}
\end{equation}
is attained on at least one curve
$\gamma_Q\in H^1\big([0,1],\overline\R^N\big)$, such that:
\begin{itemize}
\item $\big(E-V(\gamma_Q)\big)g\big(\dot\gamma_Q,\dot\gamma_Q\big)$ is
constant;
\item $\gamma_Q\big(\left]0,1\right]\big)\subset\Omega_E$, and
$\gamma_Q$ is a $C^2$ curve on $\left]0,1\right[$;
\item
$\gamma_Q$ satisfies \eqref{eq:i} of
Proposition~\ref{thm:MauJac} on the interval $[a,b]=[0,1]$.
\end{itemize}
\end{lem}
\begin{proof}
For all $k\in\N$ sufficiently large, consider the non empty open set
$\Omega_k=V^{-1}\big(\left]-\infty,E-\frac1k\right[\big)\subset\Omega_E$,
and consider the problem of minimization of the length
functional:
\begin{equation}\label{eq:Vlenght}
L_V(x)=\int_0^1\sqrt{\tfrac12(E-V(x))g(\dot x,\dot x)}\,\mathrm ds,
\end{equation}
in the space $X_k$ consisting of curves $x\in
H^1\big([0,1],\R^N\big)$ with $x(0)\in \partial \Omega_k$ and
$x(1)=Q$ and $x(]0,1])\subset\Omega_k$.

Standard arguments show that the above minimization problem has a solution
$\gamma_k$ which is a $g_*$-geodesic satisfying
$\gamma_k\big(\left]0,1\right]\big)\subset \Omega_k$ and $\gamma_k(0)\in\partial\Omega_k$.
Since $\gamma_k(0)$ approaches $\partial\Omega_E$ as $k\to\infty$, a simple contradiction argument shows that
\begin{equation}\label{eq:eq10.11}
\liminf_{k\to\infty}L_V(\gamma_k)=d_V(Q).
\end{equation}
For, if
\[\liminf_{k\to\infty} L_V(\gamma_k)> d_E(Q),\]
then there would exist a curve $x\in
H^1\big([0,1],\R^N\big)$  with
$x(0)\in\partial\Omega_E$,  $x(1)=Q$, $x(]0,1])\subset\Omega_E$, and with
$L_V(x)<\liminf\limits_{k\to\infty} L_V(\gamma_k)$.
Therefore, a suitable
reparameterization of $x$ would yield a curve $y\in X_k$
with $L_V(y)<L_V(\gamma_k)$, which contradicts the minimality of $L_V(\gamma_k)$. Hence,
\eqref{eq:eq10.11} holds.
Now, for any $s$,
\begin{multline}\label{eq:gbound}
\frac12\big(E-V(\gamma_k(s))\big)g\big(\dot x(s),\dot x(s)\big) = \\ \int_0^1\frac12\big(E-V(\gamma_k)\big)g\big(\dot x,\dot x\big)\,\mathrm d\tau
= (L_V(\gamma_k))^2
\end{multline}
while, setting
\begin{equation*}\label{eq:seq}
t_k(s)=\int_0^s\frac{L_V(\gamma_k)\mathrm d\tau}{E-V(\gamma_k(\tau))}
\end{equation*}
By Remark \ref{E-Vat-the-boundary} we have that $t_k(1)$ is bounded. Then, by \eqref{eq:gbound},  $\int_0^1g\big(\dot\gamma_k,\dot\gamma_k\big)\,\text ds$ is bounded, namely
the sequence $\gamma_k$ is bounded in
$H^1\big([0,1],\R^N\big)$. Up to subsequences, we
have a curve $\gamma_Q\in H^1\big([0,1],\R^N\big)$
which is an $H^1$-weak limit of the $\gamma_k$'s; in particular,
$\gamma_k$ is uniformly convergent to $\gamma_Q$.

We claim that $\gamma_Q$ satisfies the desired
properties. First, $\gamma_Q(]0,1])\subset\Omega_E$. Otherwise, if
$b>0$ is the last instant where $\gamma_Q(b)\in\partial\Omega_E$,
by \eqref{eq:eq10.11} and by the conservation law of the energy for
$\gamma_k$, one would have
\[
b L_V^2(\gamma_k)=\int_0^b\frac12\big(E-V(\gamma_k)\big)g\big(\dot\gamma_k,\dot\gamma_k\big)\,\text
d\tau\longrightarrow 0,
\]
because $\gamma_k$ is a minimizer, and therefore  there would exist a curve $c_k$ joining $\partial \Omega_E$ with $\gamma_k(b)$ in $\overline \Omega_k$ such that
\[
\int_0^b\frac12\big(E-V(\gamma_k)\big)g\big(\dot\gamma_k,\dot\gamma_k\big)\,\text
d\tau \leq \int_0^b\frac12\big(E-V(c_k)\big)g\big(\dot c_k,\dot c_k\big)\,\text
d\tau \longrightarrow 0.
\]
Bu then, $L_V^2(\gamma_k) \to 0$
contradicting $Q\not\in\partial\Omega_E$.

Moreover, $\gamma_Q$
satisfies \eqref{eq:i} in $[0,1]$ since it is a $H^1$--weak limit
of $\gamma_k$, which is a sequence of $g_*$--geodesics.

Clearly, $\gamma_Q$ is of class $C^2$ on $\left]0,1\right]$,
because of the convergence on each interval $\left[0,b\right]$, $b>0$.

Finally, since $L_V(\gamma_Q)\le\liminf\limits_{k\to\infty}L_V^2(\gamma_k)$, from
\eqref{eq:eq10.11} it follows that $L_V(\gamma_Q)=d_V(Q)$, and this
concludes the proof.
\end{proof}

\begin{rem}\label{rem:rem3.3}
It is immediate to see that, $\gamma_Q$ is a minimizer as in Lemma
\ref{thm:lem8.2} if and only if it is a minimizer for the functional
\begin{equation}\label{eq:funct01}
f_{0,1}(x)=\int_0^1\frac12\big(E-V(x)\big)g\big(\dot x,\dot
x\big)\,\text dt
\end{equation}
in the space of curves $X_Q$.
Then, by Lemma \ref{thm:lem8.2}, $f_{0,1}$ has at least one
minimizer on $X_Q$.
\end{rem}

Using a simple argument, we also have:
\begin{lem}\label{thm:lem84}
The map $d_V:\Omega_E\to\left[0,+\infty\right[$ defined in the
statement of Lemma~\ref{thm:lem8.2} is continuous, and it admits a
continuous extension to $\overline{\Omega}_E$ by setting $d_V=0$
on $\partial\Omega_E$.\qed
\end{lem}

\begin{lem}\label{lem:convminim}
Suppose that there is a unique minimizer $\gamma_Q$ between $V^{-1}(E)$ and $Q\in \Omega_E$. Consider $Q_n \rightarrow Q$ and let $\gamma_{Q_n}$
be a sequence of minimizers
between $V^{-1}(E)$ and $Q_n$. Then $\gamma_{Q_n} \rightarrow \gamma_Q$ in $H^1\big([0,1],\mathds{R}^N\big)$ and in $C^2([a,1],\mathds{R}^N)$ for any $a \in ]0,1[$.
\end{lem}
\begin{proof}
Consider the map
\begin{equation*}\label{eq:seqn}
t_n(s)=\int_0^s\frac{L_V(\gamma_{Q_n})\mathrm d\tau}{E-V(\gamma_{Q_n}(\tau))}.
\end{equation*}
By Remark \ref{E-Vat-the-boundary} we deduce the boundness of $t_n(1)$, so $\int_0^1 g(\dot \gamma_{Q_n},\dot \gamma_{Q_n})$ is bounded in $H^1$.
Suppose by contradiction (up to considering a subsequence) that
\begin{equation}\label{eq:nonconv}
\gamma_{Q_n} \text{ does not converge to }\gamma_Q \text{ with respect in the $H^1$ topology},
\end{equation}
and consider a subsequence $\gamma_{{Q_n}_k}$ which converges to some curve $\gamma_*$ uniformly, and weakly  in $H^1$.
Now, if $L_V$ is the functional \eqref{eq:Vlenght}, we have
\[
L_V(\gamma_*) \leq \liminf_{n \to +\infty}L_V(\gamma_{{Q_n}_k}).
\]
Now consider the minimizer $\gamma_Q$ and denote by $\hat \gamma_k$ the $H^1$--curve parametrized in $[0,1]$ joining $Q_{n_k}$ with $Q$ by a minimal $g$--geodesic
parameterized in the interval $[0,\text{dist }(Q_{n_k},Q)]$ and which coincides with the affine reparameterization of $\gamma_Q$ in the interval $[\text{dist }(Q_{n_k},Q),1]$. Clearly
\[
L_V(\gamma_{{Q_n}_k}) \leq L_V(\hat \gamma_k),
\]
and since $L_V(\hat \gamma_k) \rightarrow L_V(\gamma_Q) = d_V(Q)$ we deduce
\[
L_V(\gamma_*) \leq d_V(Q),
\]
and the uniqueness of the minimizer gives
\[
\gamma_*=\gamma_Q.
\]
Now by \eqref{eq:pr1}, $\dot \gamma_{Q_{n_k}}$ is bounded in $H^1_{loc}(]0,1],\mathds{R}^N)$, so, again by \eqref{eq:pr1}, we obtain
the convergence in $C^2([a,1])$ for any $a \in ]0,1[$.

It remains to prove that $\gamma_{{Q_n}_k}$ converges strongly in $H^1$ (to the curve $\gamma_Q$). By the weak convergence, it will
suffice to prove that
\[
\int_0^1 g(\dot \gamma_{{Q_n}_k}, \dot \gamma_{{Q_n}_k})\,\mathrm ds \longrightarrow \int_0^1 g(\dot \gamma_{Q}, \dot \gamma_{Q})\,\mathrm ds,\quad \text{as $k\to\infty$},
\]
and therefore we only need to show that
\begin{equation}\label{eq:strongconv}
t_{{n_k}}(1) \longrightarrow \int_0^1\frac{\mathrm d\tau}{E-V(\gamma_{Q}(\tau))},\quad \text{as $k\to\infty$}.
\end{equation}
To this end, consider $q_{n_k}$ and $q$, the curves obtained by the Maupertuis principle, reparameterizing $\gamma_{{Q_n}_k}$ and $\gamma_Q$ so that they satisfies \eqref{eq:eqbo}, $q_{n_k}(0)=Q_{n_k}$ and $q(0)=Q$. By the uniqueness of the minimizer $\gamma_Q$ it must be
\[
\dot q_{n_k}(0) \rightarrow \dot q(0)
\]
and continuity by the initial data in the Cauchy problem gives \eqref{eq:strongconv}, because $t_{n_k}(1)$ and $t_Q$ are uniquely determined by the relations $V(q_{n_k}(t_{n_k}))=E$ and $V(q(t_Q))=E$ respectively.
\end{proof}

\begin{prop}\label{prop:prop3.6}
Suppose that the minimizer $\gamma_Q$ between $V^{-1}(E)$ and $Q$ is unique.
Then, $d_V$ is differentiable at $Q$ and
\begin{equation}\label{eq:grad-dv}
\nabla d_V(Q)=\frac{(E-V(Q))}{2d_V(Q)}\dot\gamma_Q(1).
\end{equation}

\end{prop}

\begin{proof}
Set
\[
\psi_V(Q)=d_V(Q)^2.
\]
We have just to prove that $\psi$ is differentiable
at $Q$ and
\begin{equation}\label{eq:dpsi}
\nabla \psi_V(Q)=\big(E-V(Q)\big)\dot\gamma_Q(1).
\end{equation}
Given the local nature of the result, we can use local charts around $Q$ and
recall that $M$ is topologically embedded as an open subset of $\R^N$.
Consider $\xi \in \mathds{R}^N$ and
\[
v_\xi(s)=(2s-1)^+\xi,
\]
where $(\cdot)^+$ denotes the positive part.
Because of the behaviour of $\gamma_Q$,
for $\varepsilon$
sufficiently small (with respect to $\xi$) the curve
$\gamma_Q(s)+\varepsilon v_\xi(s)$ belongs to
$X_{Q+\varepsilon\xi}$ (see \eqref{eq:Xspace}).

Now $\gamma_Q$ is a minimizer in $X_Q$ also for
\[
\int_0^1\frac12(E-V(x))g(\dot x, \dot x)\,\mathrm ds \equiv f_{0,1}(x),
\]
so
\[
\psi(Q+\varepsilon\xi)\le f_{0,1} (\gamma_Q+\varepsilon v_\xi)
\]
and therefore
\[
\psi(Q+\varepsilon\xi)-\psi(Q)\le f_{0,1}(\gamma_Q+\varepsilon
v_\xi)-f_{0,1}(\gamma_Q).
\]
Now
\begin{multline*}
\lim\limits_{\varepsilon\to 0} \frac{1}{\varepsilon}\left(
f_{0,1}(\gamma_Q+\varepsilon v_\xi)-f_{0,1}(\gamma_Q)\right)=\\
\int_0^1\big(E-V(\gamma_Q)\big)g\big(\dot\gamma_Q,\Ddt
v_\xi\big)-\frac 12 g\big(\nabla
V(\gamma_Q),v_\xi\big)g\big(\dot\gamma_Q,\dot\gamma_Q\big)\,\text
ds
\end{multline*}
uniformly as $\Vert \xi \Vert\le 1$. Moreover, since $v_\xi=0$ in the
interval $[0,\frac 12]$, using the differential equation satisfied
by $\gamma_Q$ and integrating by parts gives
\begin{multline*}
\int_0^1\big(E-V(\gamma_Q)\big)g\big(\dot\gamma_Q,\Ddt
v_\xi\big)-\frac 12 g\big(\nabla
V(\gamma_Q),v_\xi\big)g\big(\dot\gamma_Q,\dot\gamma_Q\big)\,\text
ds=\\
\big(E-V(\gamma_Q(1))\big)g\big(\dot\gamma_Q(1),v_\xi(1)\big)=\big(E-V(Q)\big)
g\big(\dot\gamma_Q(0),\xi\big).
\end{multline*}
Therefore, uniformly as $\Vert\xi\Vert\le 1$,
\begin{equation}\label{eq:eq3.13}
\limsup_{\varepsilon\to 0^+}\frac 1\varepsilon
\left(\psi(Q+\varepsilon v_\xi)-\psi(Q)\right)-\big(E-V(Q)\big)
g(\dot\gamma_Q(1),\xi\big)\le 0.
\end{equation}
Moreover, since
$\psi(Q+\varepsilon\xi)=f_{0,1}(\gamma_{Q+\varepsilon\xi})$ and
$\psi(Q)\le f_{0,1}(\gamma_{Q+\varepsilon\xi}-\varepsilon v_\xi)$
one has
\begin{multline}\label{eq:eq3.14}
\psi(Q+\varepsilon\xi)-\psi(Q)\ge
f_{0,1}(\gamma_{Q+\varepsilon\xi})-f_{0,1}(\gamma_{Q+\varepsilon\xi}-\varepsilon
v_\xi)=\\
\varepsilon \langle
f'_{0,1}(\gamma_{Q+\epsilon\xi}),v_\xi\rangle_1-\frac{\varepsilon^2}{2}\langle
f''_{0,1}(\gamma_{Q+\varepsilon\xi}-\vartheta_\varepsilon
\varepsilon v_\xi)[v_\xi],v_\xi\rangle_1,
\end{multline}
for some $\vartheta_\varepsilon\in ]0,1[$. Here
$\langle\cdot,\cdot\rangle_1$ denotes the standard inner product
in $H^1$ and $f'$, $f''$ are respectively gradient and Hessian of $f$
with respect to $\langle\cdot,\cdot\rangle_1$.

Now, it is  $\gamma_{Q+\varepsilon\xi}(1)=Q+\varepsilon_\xi$ and
$Q\not\in V^{-1}(E)$. Moreover, by Lemma \ref{lem:convminim}, for all $l\delta>0$, there exists
$\varepsilon(\delta)>0$ such that
\[
\dist(\gamma_{Q+\varepsilon\xi}(s),\gamma_Q(s))\le\delta\qquad\text{for
any\ } \varepsilon\in ]0,\varepsilon(\delta)],\,\Vert \xi \Vert\le
1,\,s\in[0,1].
\]
Then, since $\gamma_Q$ is uniformly far from $V^{-1}(E)$ on the
interval $[\frac 12,1]$, the same holds for
$\gamma_{Q+\varepsilon\xi}$ provided that $\varepsilon$ is small and
$\Vert\xi\Vert\le 1$. Thus, recalling the definition of $d_V$ in \eqref{eq:dV} of Lemma
\ref{thm:lem8.2}, the conservation law satisfied by the minimizer
$\gamma_{Q+\varepsilon\xi}$ is
\[
\frac12\big(E-V(\gamma_{Q+\varepsilon\xi})\big)
g\big(\dot\gamma_{Q+\varepsilon\xi},\dot\gamma_{Q+\varepsilon\xi}\big)=
d_E^2(y+\varepsilon\xi).
\]
This implies the existence of a constant $C>0$ such that
\[
\int_{\frac12}^1 g\big(\dot\gamma_{Q+\varepsilon\xi},\dot\gamma_{Q+\varepsilon\xi}\big)
\,\text ds\le C
\]
for any $\varepsilon$ sufficiently small and $\Vert\xi\Vert\le 1$.

Therefore $\langle
f''_{0,1}(\gamma_{Q+\varepsilon\xi}-\vartheta_\varepsilon\varepsilon
v_\xi)[v_\xi],v_\xi\rangle_1$ is uniformly bounded with respect to
$\varepsilon$ small and $\Vert\xi\Vert\le 1$, since  $v_\xi=0$ on $[0,\frac
12]$, and by \eqref{eq:eq3.14} we get
\begin{equation}\label{eq:eq3.15}
\lim\limits_{\varepsilon\to 0}\frac
1\varepsilon\left(f_{0,1}(\gamma_{Q+\varepsilon\xi})-f_{0,1}(\gamma_{Q+\varepsilon\xi-\varepsilon
v_\xi})\right)=\lim\limits_{\varepsilon\to 0}\langle
f'_{0,1}(\gamma_{Q+\varepsilon\xi}),v_\xi\rangle_1
\end{equation}
uniformly as $\Vert\xi\Vert\le 1$.

Now, using the differential equation \eqref{eq:pr1} satisfied by
$\gamma_{Q+\varepsilon\xi}$ and integrating by parts one obtains
\[
\langle f'_{0,1}(\gamma_{Q+\varepsilon\xi}),v_\xi\rangle_1=
\big(E-V(Q+\varepsilon\xi)\big)g\big(\dot\gamma_{Q+\varepsilon\xi}(1),\xi\big),
\]
while, since the minimizer is unique, by Lemma \eqref{lem:convminim},
\begin{equation}\label{eq:eq3.16}
\lim\limits_{\varepsilon\to
0}\gamma_{Q+\varepsilon\xi}(1)=
\dot\gamma_Q(1)
\end{equation}
uniformly as $\Vert\xi\Vert\le 1$. Therefore, by
\eqref{eq:eq3.14}--\eqref{eq:eq3.16} it is
\begin{equation}\label{eq:eq3.17}
\lim\inf\limits_{\varepsilon\to 0}\frac
1\varepsilon\left(\psi(Q+\varepsilon\xi)-\psi(Q)\right)-
\big(E-V(Q)\big)g\big(\dot\gamma_y(1),\xi\big)\ge 0
\end{equation}
uniformly as $\Vert\xi\Vert\le 1$. Finally, combining \eqref{eq:eq3.13}
and \eqref{eq:eq3.17} we obtain \eqref{eq:dpsi}.
\end{proof}

\begin{rem}\label{rem:reg-near-boundary}
Since $g$ and $V$ are of class $C^2$, if $Q_0$ is sufficiently close to the boundary, then any $Q$ close
to $Q_0$ satisfies assumptions of Proposition~\ref{prop:prop3.6}, and the map $d_V$ is of class $C^2$ in a neighborhood of $Q_0$.

Indeed, denote by $q(t,x)$ the solution of the Cauchy problem \eqref{eq:eqbo} with $q(0)=x \in V^{-1}(E)$.
We have
\begin{equation}\label{eq:bo(tx)}
\dot q(t,x) = -\int_0^t\nabla V\big(q(s,x)\big)\,\mathrm ds, \;
q(t,x)=x - \int_0^t(t-s)\nabla V\big(q(s,x)\big)\,\mathrm ds,
\end{equation}
from which we deduce the $C^1$--regularity of $q(t,x)$ and $\dot q(t,x)$. Now setting $\tau=t^2$ the map
\[
q(\tau,x)=x - \int_0^{\sqrt{\tau}}(\sqrt{\tau}-\sigma)\nabla V\big(q(\sigma,x)\big)\,\mathrm d\sigma
\]
is of class $C^1$ in $(\tau,x)$, while, for every $x \in V^{-1}(E)$, $\frac{\partial q}{\partial x}(0,x) $ is the identity map and
$\frac{\partial q}{\partial \tau}(0,x)= -\frac12\nabla V(x)$. Then, $q$ is a $C^1$--diffeomorphism defined
on $\left[0,\tau_0\right[ \times V^{-1}(E)$ for a suitable $\tau_0$ sufficiently small, and  Maupertuis' principle
gives the existence of $\bar \epsilon > 0$ sufficiently small such that for any $Q$ satisfying $E-\bar\epsilon \leq V(Q) \leq E$
there exists a unique minimizer $\gamma_Q$.

We can then apply Proposition \ref{prop:prop3.6}, obtaining the differentiability of $d_V$ and formula \eqref{eq:grad-dv}.
Finally, denote by $Q\mapsto\big(\tau(Q),x(Q)\big)$ the inverse of the map $q(\tau,x)$ and observe that, by Maupertuis principle,
\[
q\big(t,x(Q)\big)=\gamma_Q(s) \text{ and }t(s)=d_V^2(Q)\int_0^s\frac{\mathrm dr}{E-V(\gamma_Q(r))},\quad s\in[0,1].
\]
Thus, by \eqref{eq:grad-dv}
\[
\nabla d_V(Q)=\tfrac12 d_V(Q) \dot q\big(\sqrt{\tau(Q)},x(Q)\big),
\]
obtaining the $C^2$--regularity of the map $d_V$.
\end{rem}

\section{The Morse Index Theorem}\label{sec:nondegen}
In order to study conjugate points and Jacobi fields for geodesics joining $V^{-1}(E)$ with $Q \in \Omega_E=V^{-1}\big(\left]-\infty,E\right[\big)$ we have to consider the geodesics action functional
\begin{equation}\label{eq:deffepsilon}
f(\gamma)=\frac12\int_0^1\big(E-V(\gamma(s))\big)g(\dot\gamma(s),\dot\gamma(s))\,\mathrm ds,
\end{equation}
which is a $C^2$-functional defined in the space $X_Q$ consisting of all absolutely continuous curves $\gamma:[0,1]\to M$
satisfying:
\begin{equation}\label{eq:H12}
 \int_0^1g(\dot\gamma,\dot\gamma)\,\mathrm dt<+\infty;
\end{equation}

\begin{equation}\label{eq:boundarycond}
\gamma(0) \in V^{-1}(E),\quad \gamma(1) = Q.
\end{equation}

The abstract analytical structure of the above variational problem is well known. The space $X_Q$ has the structure of an infinite
dimensional Hilbert manifold; for $\gamma\in X_Q$, the tangent space is identified by
\[
T_\gamma X_Q =\big\{\xi \text{ vector field of class }H^{1,2} \text{ along }\gamma:
\xi(0)\in T_{\gamma(0)}V^{-1}(E), \; \xi(1) =0\big\},
\]
(where $T_qM$ is the tangent space of $M$ at $q$),
and its natural Hilbert structure is given by
\begin{equation}\label{eq:HiSt}
\langle \xi,\xi \rangle = \int_0^1 g\big(\tfrac{\mathrm D}{\mathrm ds}\xi,\tfrac{\mathrm D}{\mathrm ds}\xi\big)\,\mathrm ds.
\end{equation}

\begin{rem}\label{rem:punti-critici}
The functional $f$ is smooth on $X_Q$, and we have:
\begin{equation}\label{eq:f'}
\mathrm df(\gamma)[\xi]=\int_0^1(E-V(\gamma))g\big(\tfrac{\mathrm D}{\mathrm ds}\xi,\dot \gamma\big) -\tfrac12g\big(\nabla V(\gamma),\xi\big)g(\dot \gamma,\dot \gamma)\,\mathrm ds,
\end{equation}
because $\dot \gamma \in L^2$. The critical points are geodesics
relatively to the metric \eqref{eq:Ejacobi} (and therefore they satisfy equation \eqref{eq:pr1}) with the boundary conditions \eqref{eq:boundarycond}.
Note that, since $Q \not\in V^{-1}(E)$, then $\gamma$ is a not constant curve. Thus, there exists a strictly positive real constant $c_\gamma$ such that
\[
\tfrac12\big(E-V(\gamma(s)\big)g\big(\dot \gamma(s), \dot \gamma(s)\big) \equiv c_\gamma
\]
In particular $\gamma(s) \in  \Omega_E$ for any $s \in ]0,1]$.

Note also that if $\gamma$ is a critical curve, partial integration in \eqref{eq:f'} does not give further conditions on $\gamma$ at $s=0$.
Indeed $E-V(\gamma(s))$ goes like $s^{2/3}$ as $s\to0$, so $\dot \gamma$ behaves like $s^{-1/3}$ as $s\to0$ and therefore
$(E-V(\gamma(s)))\dot \gamma(s)$ goes to $0$ as $s \to 0$. However there is an ``automatic" orthogonality property as pointed out in Lemma \ref{orthcond}.
\end{rem}
\medskip
If $\gamma$ is a Jacobi geodesic parameterized by arc length, the Maupertuis principle says that the relation between the arc parameter $s$ of $\gamma$ and the time that parameterized the curve $q$ is given by
\begin{equation}\label{eq:parametroTempo}
t(s)=\int_0^s \frac{1}{E-V(\gamma(r))}\mathrm dr.
\end{equation}
Since $E-V(\gamma(s))$ asymptotically behaves like $s^{2/3}$ as $s\to 0$ we immediately deduce
\begin{cor}\label{normal-behavior}
There exists a positive constant $C_0$ such that
\[
 \frac{\dot \gamma(s)}{\Vert \dot \gamma(s)\Vert} = - \frac{\nabla V(\gamma(s))}{\Vert \nabla V(\gamma(s) \Vert} + \Sigma(s), \text{ with } \Vert \Sigma(s) \Vert \leq C_0\cdot s^{\frac13}.
\]
\end{cor}

\begin{prop}\label{valutazione-hessiano}
Let $\gamma$ be a critical point of $f : X_Q \rightarrow \mathds{R}$. For any $\xi \in T_{\gamma}(X_Q)$, the Hessian $f''(\gamma)$ satisfies:
\begin{multline}\label{eq:hessiano}
f''(\gamma)[\xi,\xi]=\int_0^1-\frac12g(\dot \gamma, \dot \gamma) H^V(\gamma)[\xi,\xi] -2g(\nabla V(\gamma),\xi)g\big(\tfrac{\mathrm D}{\mathrm ds}\xi,\dot \gamma\big)\,\mathrm ds+\\
\int_0^1(E-V(\gamma))\Big[g\big(\tfrac{\mathrm D}{\mathrm ds}\xi,\tfrac{\mathrm D}{\mathrm ds}\xi\big)+g\big(R(\xi,\dot \gamma)\xi,\dot \gamma\big)\Big]\,\mathrm ds,
\end{multline}
where $R$ denotes the Riemann tensor for the metric $g$ and $H_{V}$ denotes the Hessian of $V$, namely
$\mathrm H^V(q)(v,w)=g\big((\nabla_v\nabla V)(q),w\big)$ for all $v,w\in
T_qM$ (equivalently,
$\mathrm H^\phi(q)(v,v)=\frac{\mathrm
d^2}{\mathrm ds^2}\big\vert_{s=0} \phi(\gamma(s))$, where
$\gamma:\left]-\varepsilon,\varepsilon\right[\to M$ is the unique
-- affinely parameterized -- geodesic in $M$ with $\gamma(0)=q$ and $\dot\gamma(0)=v$).
\end{prop}

\begin{rem}
Note that in our case, differently from the classical one, the initial the contribution of the geometry of the initial manifold disappears,
as if it had collapsed to a single point.
\end{rem}

\begin{proof}[Proof of Proposition~\ref{valutazione-hessiano}.]
Consider a variation  $z_r(s)$ of $\gamma$ in $X_Q$, where $r\in\left]-\epsilon,\epsilon\right[$ and $s\in[0,1]$, in such a way that $z_0=\gamma$.
We denote by $z'_r(s)$ the derivative with respect to $s$ and with $\frac{\mathrm d}{\mathrm d r} z_r$ the derivative with respect to the variational parameter $r$. Let $\xi\in T_\gamma X$ such that $\frac{\mathrm d}{\mathrm dr}\big\vert_{r=0}z_r=\xi$. Moreover,
set $h(r)=f(z_r)$, that is
\[
h(r)=\frac12\int_0^1(E-V(z_r(s))g(z'_r(s),z'_r(s))\,\mathrm ds,
\]
and since $\gamma$ is a critical point of $f$ we have
\[
f''(\gamma)[\xi,\xi]=h''(0).
\]
Differentiating $h(r)$ gives
\begin{multline*}
h'(r)=
\int_0^1\Big[-\tfrac12g(\nabla V(z_r(s),\frac{\mathrm d}{\mathrm dr}z_r(s))g(z'_r(s),z'_r(s))\\\qquad+(E-V(z_r(s))g\left(\frac{\mathrm D}{\mathrm dr} z'_r(s),z'_r(s)\right)\Big]\,\mathrm ds=\\
\int_0^1\Big[-\tfrac12g(\nabla V(z_r(s),\frac{\mathrm d}{\mathrm dr}z_r(s))g(z'_r(s),z'_r(s))\\\qquad+(E-V(z_r(s))g\left(\frac{\mathrm D}{\mathrm ds} \frac{\mathrm d}{\mathrm dr} z_r(s),z'_r(s)\right)\Big]\,\mathrm ds,
\end{multline*}
where we have used $\frac{\mathrm D}{\mathrm dr}$ and $\frac{\mathrm D}{\mathrm ds}$ to denote the covariant derivative induced by the Levi--Civita connection of $g$, made using the vector fields $\frac{\mathrm d}{\mathrm dr}z_r$ and $z'_r$ respectively. Differentiating once again we have
\begin{multline*}
h''(r)=\\\int_0^1\left[-\tfrac12 g(z'_r,z'_r)H^V(z_r)[\tfrac{\mathrm d}{\mathrm dr}z_r,\tfrac{\mathrm d}{\mathrm dr}z_r]-\tfrac12 g(z'_r,z'_r)g(\nabla V(z_r),\tfrac{\mathrm D}{\mathrm dr}\tfrac{\mathrm d}{\mathrm dr}z_r)\right]\,\mathrm ds \\
+\int_0^1 -2g(\nabla V(z_r),\tfrac{\mathrm d}{\mathrm dr}z_r)g(\tfrac{\mathrm D}{\mathrm ds}\tfrac{\mathrm d}{\mathrm dr}z_r,z'_r)\,\mathrm ds\\
+\int_0^1(E-V(z_r))\big[g(\tfrac{\mathrm D}{\mathrm dr} \tfrac{\mathrm D}{\mathrm ds}\tfrac{\mathrm d}{\mathrm dr} z_r,z'_r) + g(\tfrac{\mathrm D}{\mathrm ds}\tfrac{\mathrm d}{\mathrm dr}z_r,\tfrac{\mathrm D}{\mathrm dr} z'_r)\big]\,\mathrm ds,
\end{multline*}
where the argument $s$ in the above functions is understood.
Now
\[
\tfrac{\mathrm D}{\mathrm dr} \tfrac{\mathrm D}{\mathrm ds}\tfrac{\mathrm d}{\mathrm dr} z_r= \tfrac{\mathrm D}{\mathrm ds} \tfrac{\mathrm D}{\mathrm dr}\tfrac{\mathrm d}{\mathrm dr} z_r + R(\tfrac{\mathrm d}{\mathrm dr} z_r,z'_r)\tfrac{\mathrm d}{\mathrm dr} z_r,
\]
where $R$ denotes the Riemann tensor of $g$, chosen with the appropriate sign convention. Therefore
\begin{multline*}
h''(r)=\\\int_0^1[-\tfrac12 g(z'_r,z'_r)H^V(z_r)[\tfrac{\mathrm d}{\mathrm dr}z_r,\tfrac{\mathrm d}{\mathrm dr}z_r]-\tfrac12 g(z'_r,z'_r)g(\nabla V(z_r),\tfrac{\mathrm D}{\mathrm dr}\tfrac{\mathrm d}{\mathrm dr}z_r)]\,\mathrm ds +\\
\int_0^1 [-2g(\nabla V(z_r),\tfrac{\mathrm d}{\mathrm dr}z_r)g(\tfrac{\mathrm D}{\mathrm ds}\tfrac{\mathrm d}{\mathrm dr}z_r,z'_r) + (E-V(z))g(R(\tfrac{\mathrm d}{\mathrm dr} z_r,z'_r)\tfrac{\mathrm d}{\mathrm dr} z_r,z'_r)]\,\mathrm ds +\\
\int_0^1(E-V(z))[g(\tfrac{\mathrm D}{\mathrm ds} \tfrac{\mathrm D}{\mathrm dr}\tfrac{\mathrm d}{\mathrm dr} z_r ,z'_r) + g(\tfrac{\mathrm D}{\mathrm ds}\tfrac{\mathrm d}{\mathrm dr}z_r,\tfrac{\mathrm D}{\mathrm ds}\tfrac{\mathrm d}{\mathrm dr}z_r)]\,\mathrm ds.
\end{multline*}
By the local nature of the problem we can assume that $M$ is  an open subset of
 $\mathds{R}^N$, so we can choose
\[
z_r(s)=\gamma(s) + r\xi(s) - (1-s)[\gamma(0)+r\xi(0) - \Pi(\gamma(0)+r\xi(0))],
\]
where $\Pi$ is the projection on $V^{-1}(E)$ which is well defined for $r$ is sufficiently small (recall that the choice of $z$ is arbitrary since $\gamma$ is a critical point of $f$). Then, setting $\frac{D}{\mathrm d r}\frac{\mathrm d}{\mathrm d r}\big|_{r=0} z_r(s)=Y(s)$ we have
\begin{multline*}
h''(0)= \int_0^1-\frac12g(\dot \gamma, \dot \gamma) H^V(\gamma)[\xi,\xi] -2g(\nabla V(\gamma),\xi)g(\tfrac{\mathrm D}{\mathrm ds}\xi,\dot \gamma)\,\mathrm ds+\\
\int_0^1(E-V(\gamma))\left[g(\tfrac{\mathrm D}{\mathrm ds}\xi,\tfrac{\mathrm D}{\mathrm ds}\xi)+g(R(\xi,\dot \gamma)\xi,\dot \gamma)\right]\,\mathrm ds + \\
\int_0^1[(E-V(\gamma))g(\tfrac{\mathrm D}{\mathrm ds}Y,\dot \gamma) -\frac12 g(\dot\gamma,\dot\gamma)g(\nabla V(\gamma),Y)]\,\mathrm ds
\end{multline*}
Note that the curve $(E-V(\gamma))\dot\gamma$ can be continuously extended to $s=0$ setting $[(E-V(\gamma))\dot\gamma](0)=0$ because $g(\dot \gamma,\dot \gamma)= \frac{c_\gamma}{E-V(\gamma)}$, $c_\gamma > 0$. Moreover standard regularization arguments show that $(E-V(\gamma))\dot\gamma$ is of class $C^1$ and we deduce the Jacobi-geodesic equation
\[
-\frac{\mathrm d}{\mathrm ds}((E-V(\gamma))\dot\gamma)-\frac12g(\dot \gamma,\dot \gamma)\nabla V(\gamma) = 0,\quad \forall s \in [0,1].
\]
Partial integration gives
\[
\int_0^1[(E-V(\gamma))g(\tfrac{\mathrm D}{\mathrm ds}Y,\dot \gamma) -\frac12 g(\dot\gamma,\dot\gamma)g(\nabla V(\gamma),Y)]\,\mathrm ds= g(Y,(E-V(\gamma)\dot\gamma)\Big|_0^1.
\]
By the regularity of $z_r(s)$ we have the continuity of $y$. Finally $z_r(1)=Q$ for any $r$ gives
\[
Y(1)=\tfrac{\mathrm D}{\mathrm dr}\tfrac{\mathrm d}{\mathrm dr}\big\vert_{r=0} z_r(1)=0,
\]
while $[(E-V(\gamma))\dot\gamma](0)=0$.
\end{proof}

Now let $\gamma : [0,a] \rightarrow M$ be a Jacobi geodesic such that $\gamma(0) \in V^{-1}(E)$, parameterized by arc length, namely satisfying
\begin{equation}\label{eq:arclen}
\frac12g(\dot \gamma, \dot \gamma)(E-V(\gamma)) \equiv 1.
\end{equation}

Fix  $s \in ]0,a]$, consider the Hilbert manifold
\[
X_s\!=\!\left\{x \in AC([0,s],M): \smallint_0^s g(\dot x,\dot x)\,\mathrm d\sigma < +\infty, x(0) \in V^{-1}(E), x(s)=\gamma(s)\right\}
\]
and denote by $T_\gamma X_s $ its tangent space at $\gamma_{|[0,s]}$, endowed with the standard Hilbert structure (defined by \eqref{eq:HiSt} when $s=1$). Consider the symmetric bilinear form $I_s : T_\gamma X_s \times T_\gamma X_s \rightarrow \mathds{R}$, obtained by polarization from \eqref{valutazione-hessiano} in the interval $[0,s]$, and defined by
\begin{multline}\label{eq:Is}
I_s(\xi,\eta)= \int_0^s(E-V(\gamma))\Big[g(\tfrac{\mathrm D}{\mathrm ds}\xi,\tfrac{\mathrm D}{\mathrm ds}\eta)+g(R(\dot\gamma,\xi)\dot \gamma,\eta)\Big]\,\mathrm ds\\
-\int_0^s\Big[g(\nabla V(\gamma),\xi)g(\dot \gamma,\tfrac{\mathrm D}{\mathrm ds}\eta)+g(\dot \gamma,\tfrac{\mathrm D}{\mathrm ds}\xi)g(\nabla V(\gamma),\eta)\\
+\frac12g(H^V(\gamma)[\xi],\eta)g(\dot\gamma,\dot\gamma)\Big]\,\mathrm ds,
\end{multline}
where with a slight abuse of notation we are using the same notation $H^V(\gamma)$ for the linear application $L$ such that $g(L[\xi],\eta)=H^V(\gamma)[\xi,\eta]$. Unfortunately, due to the degeneracy of the Jacobi metric on $V^{-1}(E)$, the natural space where to study  $I_s$ is
\begin{multline}\label{eq:natspa}
Y_s =\Big \{\xi \text{ absolutely continuous vector field along }\gamma_{|[0,s]}: \\
\int_0^s(E-V(\gamma))g(\tfrac{\mathrm D}{\mathrm ds}\xi,\tfrac{\mathrm D}{\mathrm ds}\xi)\,\mathrm ds < +\infty, \;
\xi(0)\in T_{\gamma(0)}V^{-1}(E), \; \xi(s) =0\Big\},
\end{multline}
equipped with the Hilbert structure
\begin{equation}\label{eq:normapesata}
\langle \xi,\xi \rangle_{0,s} = \int_0^s(E-V(\gamma))g(\tfrac{\mathrm D}{\mathrm ds}\xi,\tfrac{\mathrm D}{\mathrm ds}\xi)\,\mathrm ds.
\end{equation}

The quadratic form associated to $I_a$, namely $I_a(\xi,\xi)$, $\xi \in Y_a$, will be denoted as the \emph{index form of $\gamma$ in $[0,a]$}.
\begin{defin}\label{eq:index}
We define the \emph{index} of $I_a$ as the maximal dimension of all subspaces of $Y_a$ on which the quadratic form $I_a(\xi,\xi)$ is negative definite. The \emph{nullity} of $I_a$ is defined to be the dimension of the subspace of $Y_a$ consisting of the elements $\xi$ such that
\[
I_a(\xi,\eta)=0 \text{ for all }\eta \in Y_a.
\]
Such a subspace is called the \emph{null space} of $I_a$.
\end{defin}

\bigskip

The null space of $I_a$ is strictly related to the Jacobi fields along $\gamma$ which are defined in the following
\begin{defin}\label{def:JacobiFields}
Let $\gamma$ be a geodesic as above. A vector field $\xi$ along $\gamma$ of class $C^0([0,a]) \cap C^2(]0,a])$ is called a \emph{Jacobi field along} $\gamma_{|[0,a]}$ if
it satisfies:
\begin{multline}\label{eq:eqkernel}
-\frac{\mathrm d}{\mathrm ds}\Big((E-V(\gamma))\tfrac{\mathrm D}{\mathrm ds}\xi\Big) + (E-V(\gamma))R(\dot \gamma,\xi)\dot \gamma+\\
+\frac{\mathrm d}{\mathrm ds}\big(g(\nabla V(\gamma),\xi)\dot \gamma\big)- g(\dot \gamma,\tfrac{\mathrm D}{\mathrm ds}\xi)\nabla V(\gamma)+\\
-\frac12g(\dot\gamma,\dot\gamma)H^{V}(\gamma)[\xi]=0 \text{ for all }s \in\left]0,a\right].
\end{multline}
\end{defin}

\begin{prop}\label{thm:nucleod2f}
$\xi \in Y_a$ is in the null space of $I_a$ if and only if it is a Jacobi field along $\gamma_{|[0,a]}$ satisfying
\begin{multline}\label{eq:alBordo}
\text{the continuous map }(E-V(\gamma))\tfrac{\mathrm D}{\mathrm ds}\xi - g(\nabla V(\gamma),\xi)\dot \gamma,\\
\text{at $0$ is parallel to }\nabla V(\gamma(0)).
\end{multline}
\end{prop}

\begin{proof}
Suppose that $\xi \in Y_a$ is in the null space of $I_a$. Then
\[
I_a(\xi,\eta)=0 \text{ for any }\eta \in Y_a.
\]
Standard regularization methods shows that $\xi \in C^2(]0,1],\mathds{R}^N)$ and using \eqref{eq:Is} after integration by parts of the quantity
\[
\int_0^a(E-V(\gamma))g(\tfrac{\mathrm D}{\mathrm ds}\xi,\tfrac{\mathrm D}{\mathrm ds}\eta)-g(\nabla V(\gamma),\xi)g(\dot \gamma,\tfrac{\mathrm D}{\mathrm ds}\eta)\,\mathrm ds
\]
gives \eqref{eq:eqkernel}. Partial integration yields also \eqref{eq:alBordo} and shows that if $\xi \in Y_a$
satisfies \eqref{eq:eqkernel} and \eqref{eq:alBordo}, then $\xi$ is in the null space of $I_a$.
\end{proof}

\begin{rem}
In analogy with the regularity of the Jacobi geodesics, one could expect that also the Jacobi fields in the null space of $I_a$ are of class $H^{1,2}$. But this is not true in general.
Consider for example the potential $V_0(x) = \frac12 \Vert x \Vert_e^2$ in $\R^N$, where $\Vert \cdot \Vert_e$ denotes the Euclidean norm. Let $g_0$ be the Euclidean metric and consider the Jacobi metric $g_* = (E-V_0)g_0$.

Fix $P$ in the unit sphere and let $\psi$ the solution of the differential equation
\[
\dot \psi\sqrt{E-\tfrac12 \psi^2} =-1
\]
in the interval $[0,a]$ such that $\psi(0)=\sqrt{2E}$. Straightforward computations shows that $\gamma(s)=\psi(s) P$ is a $g_*$--geodesic starting from the potential well, and $\xi(s)=\sqrt{E-\frac12\psi^2}P=\sqrt{E-V_0(\gamma)}$ is a Jacobi field along $\gamma$ which is in the null space of $I_a$ in $Y_a$.
Now $\dot \xi$ has the same behavior at $0$ as $s^{-\frac{2}3}$, so $\xi$ it is not in $H^{1,2}$.
\end{rem}

\begin{defin}\label{def:conjugatepoint}
The point $\gamma(s)$ is called \emph{conjugate}\footnote{For the singular case considered here, the term \emph{conjugate} to the initial manifold seems more appropriate than
the classical \emph{focal}, since there is no contribution given by the second fundamental form of the initial manifold.} to $V^{-1}(E)$ if there exists a Jacobi field $\xi \in Y_{s}\setminus\{0\}$ satisfying \eqref{eq:alBordo}. The \emph{multiplicity} of the conjugate point $\gamma(s)$ is defined as the maximum number of such linearly independent vector fields.
\end{defin}
Finally, we can state the Morse Index Theorem.
\begin{teo}\label{MIT}
The index of $I_a$ is finite and equals the number of points $\gamma(s)$, $s \in\left]0,a\right[$, conjugate to $V^{-1}(E)$ each counted with its multiplicity.
\end{teo}

\begin{rem}\label{rem:mitH12}
Since $Y_a \cap H^{1,2}$ is dense in $Y_a$, we see that the Morse Index Theorem holds also using the vector subspaces of $Y_a \cap H^{1,2}$ to define the index of $I_a$, provided that we continue to use  Jacobi fields lying in $Y_a$ to define the conjugate points.
\end{rem}

\begin{rem}
Note that by Theorem \ref{MIT} we see that there is only a finite number of points conjugate to $V^{-1}(E)$.
\end{rem}

\medskip

To prove the Morse Index Theorem we need some preliminary results. The subtler
one is Proposition~\ref{thm:minimizer}, which deals with the existence of a minimizer for the quadratic form
\[
F_s(\xi)=I_s(\xi,\xi)
\]
($s \in ]0,a]$) on the space
\begin{multline}\label{eq:YW}
Y_s^W = \Big\{\xi \text{ absolutely continuous vector field along }\gamma\big\vert_{[0,s]}: \\
\int_0^s(E-V(\gamma))g(\tfrac{\mathrm D}{\mathrm ds}\xi,\tfrac{\mathrm D}{\mathrm ds}\xi)\,\mathrm ds < +\infty, \;
\xi(0)\in T_{\gamma(0)}V^{-1}(E), \; \xi(s) =W\Big\}
\end{multline}
where $W \in T_{\gamma(s)}M$.

To prove the Morse index Theorem we shall consider also the quadratic form $I_{s_1,s_2}$ which is just the integral \eqref{eq:Is} in the interval $[s_1,s_2]$ ($0< s_1 < s_2 \leq a$), defined on the vector space
\begin{multline}\label{eq:YW2}
Y_{s_1,s_2}^{W_1,W_2} = \Big\{\xi \text{ absolutely continuous vector field along }\gamma\big\vert_{[s_1,s_2]}: \\
\int_{s_1}^{s_2}g(\tfrac{\mathrm D}{\mathrm ds}\xi,\tfrac{\mathrm D}{\mathrm ds}\xi)\,\mathrm ds < +\infty, \;
\xi(s_1)=W_1, \xi(s_2) =W_2\Big\}
\end{multline}
where $W_i \in T_{\gamma(s_1)}M, i=1,2$.

\begin{rem}\label{rem:stima-uniforme}
If $\int_0^s(E-V(\gamma))g(\tfrac{\mathrm D}{\mathrm ds}\xi,\tfrac{\mathrm D}{\mathrm ds}\xi)\,\mathrm ds < +\infty$, then $\Vert \xi \Vert_\infty < +\infty$ where $\Vert \cdot \Vert_\infty$ is the $L^\infty$-norm. Indeed, denoting by $d_g$ the distance induced by the Riemannian structure $g$ and by $\Vert \cdot \Vert$ the norm induced by $g$ in the tangent space, we have, for any $0 < s_1 < s_2 \leq a$,
\begin{multline*}
d_g(\xi(s_1),\xi(s_2)) \leq \int_{s_1}^{s_2} \frac{\sqrt{E-V(\gamma)}\Vert \tfrac{\mathrm D}{\mathrm ds}\xi \Vert}{\sqrt{E-V(\gamma)}}\,\mathrm ds\leq\\
\Big(\int_{s_1}^{s_2}(E-V(\gamma)) \Vert \tfrac{\mathrm D}{\mathrm ds}\xi \Vert^2 \,\mathrm ds\Big)^{1/2}\cdot\Big( \int_{s_1}^{s_2}\Vert \dot \gamma \Vert^2 \,\mathrm ds\Big)^{1/2},
\end{multline*}
while $\Vert \dot \gamma \Vert$ goes like $s^{-1/3}$ as $s\to 0^+$.
\end{rem}
In order to prove the existence of a minimizer, the following Lemma will be useful:
\begin{lem}\label{lem:xisulbordo}
If $\xi \in Y_s^W$ then
\begin{equation}\label{eq:assertion}
\lim_{s\rightarrow 0^{+}}g\big(\nabla V(\gamma),\xi\big)\,g(\dot\gamma,\xi)=0.
\end{equation}
\end{lem}
\begin{proof}
First observe that, by Corollary  \ref{normal-behavior}
\begin{multline*}
g(\nabla V(\gamma),\xi)g(\dot \gamma,\xi)=
g(\nabla V(\gamma),\xi)g\left(\xi, -\frac{\nabla V(\gamma)}{\Vert \nabla V(\gamma) \Vert}\right)\Vert \dot \gamma \Vert +\\
g(\nabla V(\gamma),\xi)g(\xi,\Sigma)\Vert \dot \gamma \Vert.
\end{multline*}
Now
\[
\lim_{s\rightarrow 0^{+}} g(\nabla V(\gamma),\xi)g(\xi,\Sigma)\Vert \dot \gamma \Vert = 0,
\]
since $\Vert \Sigma \Vert \leq C_0s^{1/3}$, $\Vert \dot \gamma \Vert \leq C_\gamma s^{-1/3}$ for some $C_\gamma>0$, and
$g(\nabla V(\gamma(0)),\xi(0))=0$. Moreover, for some positive constants $d_1,d_2$,
\begin{multline*}
|g(\nabla V(\gamma),\xi)(s)|\leq \int_0^s |g(H^V(\gamma)[\dot \gamma],\xi)|\,\mathrm d\sigma +  \int_0^s |g(\nabla V(\gamma),\tfrac{\mathrm D}{\mathrm d\sigma}\xi)|\,\mathrm d\sigma \leq\\
d_1s^{2/3}\Vert \xi \Vert_\infty  + d_2\Big(\int_{0}^{s}\sqrt{E-V(\gamma)} \Vert \tfrac{\mathrm D}{\mathrm ds}\xi \Vert^2 \,\mathrm ds\Big)^{1/2}\Big( \int_{0}^{s}\sigma^{-2/3}\,\mathrm d\sigma\Big)^{1/2},
\end{multline*}
so
\[
|g(\nabla V(\gamma),\xi)(s)|^2\Vert \dot \gamma \Vert \leq\\
 2d_1^2s\Vert \xi \Vert_\infty^2  + 2d_2\int_{0}^{s}(E-V(\gamma))g(\tfrac{\mathrm D}{\mathrm d\sigma}\xi,\tfrac{\mathrm D}{\mathrm d\sigma}\xi)\,\mathrm d\sigma,
\]
from which we deduce \eqref{eq:assertion}, because the map $\big(E-V(\gamma)\big)g\big(\frac{\mathrm D}{\mathrm ds}\xi,\frac{\mathrm D}{\mathrm ds}\xi\big)$ is in $L^1$.
\end{proof}

\begin{prop}\label{thm:minimizer}
There exists $\hat s \in\left]0,a\right]$ such that, for all $s_* \in\left]0,\hat s\right]$ and for all $W \in T_{\gamma_{s_*}}M$,
the functional $F_{s_*}$ has a minimizer in $Y_{s_*}^{W}$, which is a Jacobi field along $\gamma\vert_{[0,s_*]}$.
\end{prop}

\begin{proof}
First let us recall from \eqref{eq:Is} that
\begin{multline}\label{eq:Is2}
I_{s_*}(\xi,\xi)= \int_0^{s_*}(E-V(\gamma))\Big[\Vert\tfrac{\mathrm D}{\mathrm ds}\xi\Vert^2+g(R(\dot\gamma,\xi)\dot \gamma,\xi)\Big]\,\mathrm ds\\
-\int_0^{s_*}2\Big[g(\nabla V(\gamma),\xi)g(\dot \gamma,\tfrac{\mathrm D}{\mathrm ds}\xi)
+\frac12g(H^V(\gamma)[\xi],\xi)\Vert\dot\gamma\Vert^2\Big]\,\mathrm ds,
\end{multline}

Let us now estimate some of the terms in the above expression.
First, let $C_1$ be a constant such that
\begin{equation}\label{eq:est2}
\left|\int_0^{s_*} g(R(\dot\gamma,\xi)\dot\gamma,\xi)-\frac12 g(H^V(\gamma)[\xi],\xi)\Vert\dot\gamma\Vert^2\,\mathrm ds\right|\le\frac{C_1}{2}\int_0^{s_*}\Vert\dot\gamma\Vert^2\,\Vert\xi\Vert^2\,\mathrm ds
\end{equation}
Now we want to estimate the first term in the second line of \eqref{eq:Is2} above. To this end, we observe that, using integration by parts, recalling that $\xi(s_*)=W$ and $\gamma$ satisfies \eqref{eq:pr1}, thanks to Lemma \ref{lem:xisulbordo}  we have for any $\xi \in Y_{s_*}^{W}$
\begin{multline}\label{eq:stima-per-parti-2}
-\int_{0}^{s_*}g(\nabla V(\gamma),\xi)g(\tfrac{\mathrm D}{\mathrm ds}\xi,\dot \gamma)\,\mathrm ds - \int_{0}^{s_*}g(\nabla V(\gamma),\tfrac{\mathrm D}{\mathrm ds}\xi)g(\xi,\dot \gamma)\,\mathrm ds =\\
-g(\nabla V(\gamma(s_*)),W)g(\dot \gamma(s_*),W) +
\int_{0}^{s_*}g(H^V(\gamma)[\dot \gamma],\xi)g(\dot \gamma,\xi)\,\mathrm ds +  \\
\int_{0}^{s_*}\frac{g(\nabla V(\gamma),\xi)}{E-V(\gamma)}\Big(g(\nabla V(\gamma),\dot \gamma)g(\dot \gamma,\xi)-\frac12g(\dot \gamma, \dot \gamma)
g(\nabla V(\gamma),\xi) \Big)\,\mathrm ds.
\end{multline}

On the other side, using Corollary \ref{normal-behavior}, one obtains
\begin{multline}\label{eq:stima-3}
-\int_{0}^{s_*}g(\nabla V(\gamma),\xi)g(\tfrac{\mathrm D}{\mathrm ds}\xi,\dot \gamma)\,\mathrm ds - \int_{0}^{s_*}g(\nabla V(\gamma),\tfrac{\mathrm D}{\mathrm ds}\xi)g(\xi,\dot \gamma)\,\mathrm ds =\\
-2\int_{0}^{s_*}g(\nabla V(\gamma),\xi)g(\tfrac{\mathrm D}{\mathrm ds}\xi,\dot \gamma)\,\mathrm ds+\\
\int_{0}^{s_*}g(\nabla V(\gamma),\xi)g(\tfrac{\mathrm D}{\mathrm ds}\xi,\Sigma)\Vert \dot \gamma \Vert\,\mathrm  ds - \int_{0}^{s_*}g(\nabla V(\gamma),\tfrac{\mathrm D}{\mathrm ds}\xi)
g(\xi,\Sigma)\Vert \dot \gamma \Vert\,\mathrm  ds,
\end{multline}
and then equating the righthand sides of \eqref{eq:stima-per-parti-2} and \eqref{eq:stima-3} we get
\begin{multline}\label{eq:est3}
-2\int_{0}^{s_*}g(\nabla V(\gamma),\xi)g(\tfrac{\mathrm D}{\mathrm ds}\xi,\dot \gamma)\,\mathrm ds=\\
-g(\nabla V(\gamma(s_*)),W)g(\dot \gamma(s_*),W) +
\int_{0}^{s_*}g(H^V(\gamma)[\dot \gamma],\xi)g(\dot \gamma,\xi)\,\mathrm ds   \\
+\int_{0}^{s_*}\frac{g(\nabla V(\gamma),\xi)}{E-V(\gamma)}\Big(g(\nabla V(\gamma),\dot \gamma)g(\dot \gamma,\xi)-\frac12g(\dot \gamma, \dot \gamma)
g(\nabla V(\gamma),\xi) \Big)\,\mathrm ds\\
-\int_{0}^{s_*}\left(g(\nabla V(\gamma),\xi)g(\tfrac{\mathrm D}{\mathrm ds}\xi,\Sigma) -g(\nabla V(\gamma),\tfrac{\mathrm D}{\mathrm ds}\xi)
g(\xi,\Sigma)\right)\Vert \dot \gamma \Vert\,\mathrm  ds.
\end{multline}
We can now estimate the righthand side above as follows. First, let $C_2=C_2(s_*)$ be a constant such that
\begin{equation}\label{eq:est4}
\left|g(\nabla V(\gamma(s_*)),W)g(\dot \gamma(s_*),W)\right|\le C_2.
\end{equation}
Secondly, there exists a constant -- that we can assume equal to $C_1$ as in \eqref{eq:est2} -- such that
\begin{equation}\label{eq:est5}
\left|\int_{0}^{s_*}g(H^V(\gamma)[\dot \gamma],\xi)g(\dot \gamma,\xi)\,\mathrm ds\right|\le
\frac{C_1}{2}\int_0^{s_*}\Vert\dot\gamma\Vert^2\,\Vert\xi\Vert^2\,\mathrm ds.
\end{equation}
Since $E-V(\gamma(s))$ behaves like $s^{2/3}$ as $s\to 0$, by Corollary \ref{normal-behavior},  we immediately obtain the existence of a constant $\tilde C_0$ such that
\begin{equation}\label{stima-Sigma}
\Vert \Sigma \Vert \leq \tilde C_0\sqrt{E-V(\gamma(s))} \text{ for any }s.
\end{equation}
Then there exists a constant $C_£$ such that
\begin{multline}\label{eq:est7}
\left|\int_{0}^{s_*}\left(g(\nabla V(\gamma),\xi)g(\tfrac{\mathrm D}{\mathrm ds}\xi,\Sigma) -g(\nabla V(\gamma),\tfrac{\mathrm D}{\mathrm ds}\xi)
g(\xi,\Sigma)\right)\Vert \dot \gamma \Vert\,\mathrm  ds\right| \\
\leq C_3 \Vert \xi \Vert_{\infty} \Big(\int_{0}^{s_*}(E-V(\gamma)) \Vert \tfrac{\mathrm D}{\mathrm ds}\xi \Vert^2 \,\mathrm ds\Big)^{1/2}
\Big( \int_{0}^{s_*}\Vert \dot \gamma \Vert^2\,\mathrm  ds\Big)^{1/2},
\end{multline}
where $\Vert \cdot \Vert_{\infty}$ denotes the norm in $L^\infty$.

Finally, to estimate the remaining part in \eqref{eq:est3}, observe that by Corollary \ref{normal-behavior} again we get
\begin{multline*}
g(\nabla V(\gamma),\dot \gamma)g(\dot \gamma,\xi)-\tfrac12g(\dot \gamma, \dot \gamma)
g(\nabla V(\gamma),\xi) =\\
\Vert\dot\gamma\Vert^2\Big[\tfrac12g(\nabla V(\gamma),\xi)-\Vert\nabla V(\gamma)\Vert g(\Sigma,\xi)\\
+g(\nabla V(\gamma),\Sigma) g(\xi,\Sigma)-\frac12\Vert\Sigma\Vert^2g(\nabla V(\gamma),\xi)\Big],
\end{multline*}
obtaining
\begin{multline}\label{eq:est6}
\int_{0}^{s_*}\frac{g(\nabla V(\gamma),\xi)}{E-V(\gamma)}\Big(g(\nabla V(\gamma),\dot \gamma)g(\dot \gamma,\xi)-\frac12g(\dot \gamma, \dot \gamma)
g(\nabla V(\gamma),\xi) \Big)\,\mathrm ds=\\
\int_0^{s_*}\frac{g(\nabla V(\gamma),\xi)^2}{2(E-V(\gamma))}\Vert\dot\gamma\Vert^2\,\mathrm ds
+\int_0^{s_*}
\frac{g(\nabla V(\gamma),\xi)}{(E-V(\gamma))}\cdot \\
\cdot \left(-\Vert\nabla V(\gamma)\Vert g(\Sigma,\xi)+g(\nabla V(\gamma),\Sigma) g(\xi,\Sigma)-\frac12\Vert\Sigma\Vert^2 g(\nabla V(\gamma),\xi)\right)\Vert\dot\gamma\Vert^2\,\mathrm ds,
\end{multline}
and now observe that the first integral in the righthand side above,
since by \eqref{eq:arclen} $\Vert\dot\gamma\Vert^2=2(E-V(\gamma))^{-1}$, can be written as
\begin{equation}\label{eq:est6ter}
\int_0^{s_*}\frac{g(\nabla V(\gamma),\xi)^2}{2(E-V(\gamma))}\Vert\dot\gamma\Vert^2\,\mathrm ds=
\int_0^{s_*}\left(\frac{g(\nabla V(\gamma),\xi)}{E-V(\gamma)}\right)^2\,\mathrm ds
\end{equation}
while the second integral in the righthand side of \eqref{eq:est6}, (using \eqref{stima-Sigma} to estimate the infinitesimal quantity $\Sigma$ and recalling  $\Vert\dot\gamma\Vert^2=2(E-V(\gamma))^{-1}$ again), can be estimated in norm by the quantity
\begin{equation}\label{eq:est6bis}
C_4\int_0^{s_*}\frac{|g(\nabla V(\gamma),\xi)|}{(E-V(\gamma))}\Vert\xi\Vert\,\Vert \dot \gamma\Vert\,\mathrm ds
\end{equation}
for some suitable constant $C_4$.

Therefore, joining together in \eqref{eq:Is2} information from \eqref{eq:est2} and \eqref{eq:est3}--\eqref{eq:est6bis}, we can control  $I_{s_*}(\xi,\xi)$ from below as follows:
\begin{multline}\label{eq:estIs2}
I_{s_*}(\xi,\xi)\ge \int_0^{s_*}\Big[(E-V(\gamma))\left\Vert\tfrac{\mathrm D}{\mathrm ds}\xi\right\Vert^2+\left(\frac{g(\nabla V(\gamma),\xi)}{E-V(\gamma)}\right)^2\Big]\,\mathrm ds\\
-C_2-C_1\int_0^{s_*}\Vert\dot\gamma\Vert^2\,\Vert\xi\Vert^2\,\mathrm ds
-C_4\int_0^{s_*}\frac{|g(\nabla V(\gamma),\xi)|}{(E-V(\gamma))}\Vert\xi\Vert\,\Vert\dot \gamma \Vert\,\mathrm ds\\
-C_3\Vert \xi \Vert_{\infty} \Big(\int_{0}^{s_*}(E-V(\gamma)) \Vert \tfrac{\mathrm D}{\mathrm ds}\xi \Vert^2 \,\mathrm ds\Big)^{1/2}
\Big( \int_{0}^{s_*}\Vert \dot \gamma \Vert^2\,\mathrm  ds\Big)^{1/2}.
\end{multline}
Now our aim is to find two positive constant values $\delta$ and  $A$ depending on $s_*$ such that
\begin{equation}\label{eq:stimaIs}
I_{s_*}(\xi,\xi)\ge\delta\int_0^{s_*}\Big[(E-V(\gamma))\left\Vert\tfrac{\mathrm D}{\mathrm ds}\xi\right\Vert^2+\left(\frac{g(\nabla V(\gamma),\xi)}{E-V(\gamma)}\right)^2\Big]\,\mathrm ds-A(s_*).
\end{equation}
To prove this fact, the terms on the second and third row in \eqref{eq:estIs2} must be conveniently estimated.
As an example, we show the argument for the term which is multiplied by $C_1$. First observe that, thanks to Remark \ref{rem:stima-uniforme}, we can write
\begin{equation}\label{eq:xinfty}
\Vert\xi\Vert_\infty\le\Vert W\Vert
+\Big(\int_{0}^{s_*}(E-V(\gamma)) \left\Vert \tfrac{\mathrm D}{\mathrm ds}\xi\right\Vert^2 \,\mathrm ds\Big)^{\frac12}\Big( \int_{0}^{s_*}\Vert \dot \gamma \Vert^2 \,\mathrm ds\Big)^{\frac12},
\end{equation}
and then
\begin{multline*}
\int_0^{s_*}\Vert\dot\gamma\Vert^2\,\Vert\xi\Vert^2\,\mathrm ds \\
\le\left[\Vert W\Vert
+\Big(\int_{0}^{s_*}(E-V(\gamma)) \left\Vert \tfrac{\mathrm D}{\mathrm ds}\xi\right\Vert^2 \,\mathrm ds\Big)^{\frac12}\Big( \int_{0}^{s_*}\Vert \dot \gamma \Vert^2 \,\mathrm ds\Big)^{\frac12}\right]^2\int_{0}^{s_*}\Vert \dot \gamma \Vert^2 \,\mathrm ds\\
\le 2\Vert W\Vert^2\int_{0}^{s_*}\Vert \dot \gamma \Vert^2 \,\mathrm ds+2\int_{0}^{s_*}(E-V(\gamma)) \left\Vert \tfrac{\mathrm D}{\mathrm ds}\xi\right\Vert^2 \,\mathrm ds\int_{0}^{s_*}\Vert \dot \gamma \Vert^2 \,\mathrm ds.
\end{multline*}
Since $\Vert \dot \gamma \Vert^2$ is in $L^1([0,s_*]$,  we can choose $s_*$ in such a way that
$$
-C_1\int_0^{s_*}\Vert\dot\gamma\Vert^2\,\Vert\xi\Vert^2\,\mathrm ds \ge -A_1-\delta_1 \int_{0}^{s_*}(E-V(\gamma)) \left\Vert \tfrac{\mathrm D}{\mathrm ds}\xi\right\Vert^2 \,\mathrm ds
$$
with $\delta_1>0$ that can made arbitrarily small choosing $s_*$ small enough. Likewise, all the other terms in the second and third row of \eqref{eq:estIs2} can be estimated in order to obtain \eqref{eq:stimaIs}.

Now let $\xi_n$ be a minimizing sequence. First note that by \eqref{eq:stimaIs}, the quantity
\[
\int_0^{s_*}\Big[(E-V(\gamma))\left\Vert\tfrac{\mathrm D}{\mathrm ds}\xi\right\Vert^2+\left(\frac{g(\nabla V(\gamma),\xi)}{E-V(\gamma)}\right)^2\Big]
\]
is bounded. Then, up to taking subsequences, we can assume the weak convergence $\xi_n \rightharpoonup \xi$ in $H^1_{loc}(]0,s_*])$, and the uniform convergence of $\xi_n$ to $\xi$ in $[0,s_*]$ (thanks to Remark \ref{rem:stima-uniforme}), from which we deduce that $\xi \in Y_{s*}^W$ and
\begin{multline*}
\int_0^{s_*}\Big[(E-V(\gamma))\left\Vert\tfrac{\mathrm D}{\mathrm ds}\xi\right\Vert^2+\left(\frac{g(\nabla V(\gamma),\xi)}{E-V(\gamma)}\right)^2\Big]\\
\leq \liminf_{n\to +\infty}
\int_0^{s_*}\Big[(E-V(\gamma))\left\Vert\tfrac{\mathrm D}{\mathrm ds}\xi_n\right\Vert^2+\left(\frac{g(\nabla V(\gamma),\xi_n)}{E-V(\gamma)}\right)^2\Big].
\end{multline*}
Moreover, since $\xi_n$ is uniformly convergent to $\xi$ and $\Vert \dot \gamma \Vert^2$ is in $L^1([0,s_*]$ we have
\[
\int_0^{s_*}(E-V(\gamma))g(R(\dot\gamma,\xi_n)\dot \gamma,\xi_n)\mathrm ds \rightarrow \int_0^{s_*}(E-V(\gamma))g(R(\dot\gamma,\xi)\dot \gamma,\xi)\mathrm ds
\]
and
\[
\int_{0}^{s_*}g(H^V(\gamma)[\dot \gamma],\xi_n)g(\dot \gamma,\xi_n)\,\mathrm ds\rightarrow \int_{0}^{s_*}g(H^V(\gamma)[\dot \gamma],\xi)g(\dot \gamma,\xi)\,\mathrm ds.
\]
Now, estimate \eqref{eq:est7} can be obtained for any $\bar s \in ]0,s_*]$, and we can choose $\bar s$ so that
\[
\int_{0}^{\bar s}\left(g(\nabla V(\gamma),\xi_n)g(\tfrac{\mathrm D}{\mathrm ds}\xi_n,\Sigma) -g(\nabla V(\gamma),\tfrac{\mathrm D}{\mathrm ds}\xi_n)
g(\xi_n,\Sigma)\right)\Vert \dot \gamma \Vert\,\mathrm  ds
\]
is arbitrarily small,
because $\int_{0}^{s_*}(E-V(\gamma)) \Vert \tfrac{\mathrm D}{\mathrm ds}\xi_n \Vert^2 \,\mathrm ds$ is equi-bounded and $\Vert \dot \gamma \Vert^2$ is in $L^1$.
Moreover the week convergence of $\xi_n$ in $H^1([\bar s,s_*])$ gives
\begin{multline*}
\int_{\bar s}^{s_*}\left(g(\nabla V(\gamma),\xi_n)g(\tfrac{\mathrm D}{\mathrm ds}\xi_n,\Sigma) -g(\nabla V(\gamma),\tfrac{\mathrm D}{\mathrm ds}\xi_n)
g(\xi_n,\Sigma)\right)\Vert \dot \gamma \Vert\,\mathrm  ds \rightarrow \\
\int_{\bar s}^{s_*}\left(g(\nabla V(\gamma),\xi)g(\tfrac{\mathrm D}{\mathrm ds}\xi,\Sigma) -g(\nabla V(\gamma),\tfrac{\mathrm D}{\mathrm ds}\xi)
g(\xi,\Sigma)\right)\Vert \dot \gamma \Vert\,\mathrm  ds.
\end{multline*}
Analogously we can choose $\bar s$ so that
\begin{multline*}
\int_{0}^{\bar s}\frac{g(\nabla V(\gamma),\xi_n)}{(E-V(\gamma))}\Big(-\Vert\nabla V(\gamma)\Vert g(\Sigma,\xi_n)+g(\nabla V(\gamma),\Sigma) g(\xi_n,\Sigma)+ \\
-\frac12\Vert\Sigma\Vert^2 g(\nabla V(\gamma),\xi_n)\Big)\Vert\dot\gamma\Vert^2\,\mathrm ds
\end{multline*}
is arbitrarily small, and use the uniform convergency of $\xi_n$ to obtain that
\begin{multline*}
 \int_{\bar s}^{s_*}\frac{g(\nabla V(\gamma),\xi_n)}{(E-V(\gamma))}\Big(-\Vert\nabla V(\gamma)\Vert g(\Sigma,\xi_n)+g(\nabla V(\gamma),\Sigma) g(\xi_n,\Sigma)+ \\
-\frac12\Vert\Sigma\Vert^2 g(\nabla V(\gamma),\xi_n)\Big)\Vert\dot\gamma\Vert^2\,\mathrm ds \rightarrow \\
\int_{\bar s}^{s_*}\frac{g(\nabla V(\gamma),\xi)}{(E-V(\gamma))}\Big(-\Vert\nabla V(\gamma)\Vert g(\Sigma,\xi)+g(\nabla V(\gamma),\Sigma) g(\xi,\Sigma)+ \\
-\frac12\Vert\Sigma\Vert^2 g(\nabla V(\gamma),\xi)\Big)\Vert\dot\gamma\Vert^2\,\mathrm ds.
\end{multline*}
In a such a way, by  \eqref{eq:est3},
\[
\int_{0}^{s_*}g(\nabla V(\gamma),\xi_n)g(\tfrac{\mathrm D}{\mathrm ds}\xi_n,\dot \gamma)\,\mathrm ds \rightarrow
\int_{0}^{s_*}g(\nabla V(\gamma),\xi)g(\tfrac{\mathrm D}{\mathrm ds}\xi,\dot \gamma)\,\mathrm ds
\]
and we are done.
\end{proof}

\begin{prop}\label{prop:s-0unico}
Let  $\hat s$ be as in Proposition \ref{thm:minimizer} and $s_* \in ]0,\hat s]$. Then for any $W \in T_{\gamma(s_*)}M$, $I_{s_*}$ is strictly positive definite and there exists a unique Jacobi field in $Y_{s_*}^W$.
\end{prop}
\begin{proof}
Denote by $J_{s_*}$ the linear space consisting of the Jacobi fields $\xi$ such that $\langle \xi,\xi \rangle_{0,s_*} < +\infty$ (cf \eqref{eq:normapesata}) and $\xi(0) \in T_{\gamma(0)} V^{-1}(E)$. Consider the linear map $L:J_{s_*}\rightarrow T_{\gamma(s_*)}M$ such that $L(\xi) = \xi(s_*)$.
By Proposition \ref{thm:minimizer} $L$ is surjective. To prove injectivity observe that, if $W=0$, by the same proof of Proposition \ref{thm:minimizer} we obtain the existence of $\delta_0 > 0$ such that (if $\hat s$ is sufficiently small)
\[
I_{s_*}(\xi,\xi) \geq \delta_0 \langle \xi,\xi \rangle_{0,s_*}, \text{ for any }\xi \in Y_{s_*}^0,
\]
proving that there is a unique vector field in the null space of $I_{s_*}$, namely the null Jacobi field.
\end{proof}

Since the Jacobi metric along a geodesic moving from the boundary of potential well is degenerate only at the starting point, using standard estimates for the Hessian of the action integral written in terms of the Jacobi metric, we obtain also the following
\begin{prop}\label{prop:estremi fissi}
For any $\delta > 0$ there exists $\epsilon > 0$ such that, if $\delta \leq s_1 < s_2 \leq a$,  $s_2 - s_1 \leq \epsilon$ and $W_1 \in  T_{\gamma(s_1)}M, W_2 \in T_{\gamma(s_2)}M$, the quadratic form $I_{s_1,s_2}$ is strictly positive definite and there exists a unique minimizer. It is a Jacobi field $\xi$ (of class $H^{1,2}$ along $\gamma_{|[s_1,s_2]}$) such that $\xi(s_1)=W_1, \xi(s_2)=W_2$.
\end{prop}

Thanks to Propositions \ref{prop:s-0unico} and  \ref{prop:estremi fissi} we can choose a subdivision
\[
0=s_0 < s_1 < \ldots < s_{k-1} < s_k =a
\]
such that $s_1 \leq \hat s$ and, for any $i=2,\ldots,k$, $s_i-s_{i-1} \leq \epsilon$, where $\delta=s_1$.
Denote by $\mathcal V$ the space of the vector fields along $\gamma_{|[0,a]}$ of class $H^{1,2}$ and such that
\begin{equation}\label{eq:in0s1}
\langle \xi,\xi \rangle_{0,s_1} < +\infty, \; g(\xi(0),\nabla V(\gamma(0)) = 0.
\end{equation}

Moreover denote by $\mathcal V^-$ the finite--dimensional vector subspace of $\mathcal V$ consisting of the vector fields $\xi$ along $\gamma_{[0,a]}$ satisfying \eqref{eq:alBordo} and such that
for any $i=1,\ldots,k$, $\xi_{|[s_{i-1},s_i]}$ is a Jacobi field along $\gamma_{|[s_{i-1},s_i]}$. Moreover denote by $\mathcal V^+$ the vector subspace of $\mathcal V$ consisting of the vector fields $\eta$ along $\gamma$ such that
\[
\langle \eta,\eta \rangle_{0,s_1} < +\infty, \quad \eta(s_1)=\eta(s_2)=\ldots=\eta(s_{k-1})=0.
\]
\begin{prop}\label{prop:sommadiretta}
$\mathcal V$ is direct sum $\mathcal V ={\mathcal V}^+ \oplus {\mathcal V}^-$, and the subspace ${\mathcal V}^+$ and ${\mathcal V}^-$ are orthogonal with respect to $I_a$. In addition, $I_{a}$ restricted to ${\mathcal V}^+$ is positive definite.
\end{prop}
\begin{proof}
Let $\eta \in \mathcal V$, and $\xi \in \mathcal V^-$ such that $\xi(s_j)=\eta(s_j)$ for any $j=1,\ldots,k-1$. Note that by Propositions \ref{prop:s-0unico} and  \ref{prop:estremi fissi} we see that such a $\xi$ exists and it is unique, from which we deduce that $\mathcal V ={\mathcal V}^+ \oplus {\mathcal V}^-$.
Moreover,
\begin{multline*}
I_a(\xi,\eta)= \sum_{i=1}^{k}\Big( \int_{s_{i-1}}^{s_i} (E-V(\gamma))\Big[g(\tfrac{\mathrm D}{\mathrm ds}\xi,\tfrac{\mathrm D}{\mathrm ds}\eta)+g(R(\dot\gamma,\xi)\dot \gamma,\eta)\Big]\,\mathrm ds+\\
-\int_{s_{i-1}}^{s_i}\Big[g(\nabla V(\gamma),\xi)g(\dot \gamma,\tfrac{\mathrm D}{\mathrm ds}\eta)+g(\dot \gamma,\tfrac{\mathrm D}{\mathrm ds}\xi)g(\nabla V(\gamma),\eta)+\\
\frac12g(H^V(\gamma)[\xi],\eta)g(\dot\gamma,\dot\gamma)\Big]\,\mathrm ds\Big),
\end{multline*}
so integrating by parts gives $I_a(\xi,\eta)=0$ for any $\xi \in \mathcal V^{-}$, for any $\eta \in \mathcal V^{+}$. Finally thanks to the definition of $\mathcal V^+$ and Propositions \ref{prop:s-0unico} and  \ref{prop:estremi fissi} we deduce also that $I_a$ is strictly positive definite on $\mathcal V^+$.
\end{proof}

\begin{cor}
The index of $I_a$ is equal to the index of $I_a$ restricted to $\mathcal V^-$. In particular the index of $I_a$ is finite. And the same result also holds for the nullity of $I_a$.
\end{cor}

\begin{proof}[Proof of Theorem \ref{MIT}]
Propositions  \ref{prop:s-0unico},  \ref{prop:estremi fissi} and \ref{prop:sommadiretta} allows to repeat the proof of the Index Theorem given in
\cite{docarmo} (where the case with fixed extreme points is considered). Here we  only observe that, to prove  \cite[Lemma 2.8]{docarmo},  it is more convenient to use Proposition \ref{prop:estremi fissi}, rather the Index Lemma in  \cite[Chapter 10]{docarmo}.
\end{proof}

\begin{rem}\label{rem:monotonia}
Proposition \ref{prop:estremi fissi} (rather than the Index Lemma of \cite[Ch.\ 10]{docarmo}) is used to prove that, denoting by $\nu_s$ the nullity of $I_s$ and by $i(s)$ its index, for any $\epsilon > 0$ sufficiently small
\[
i(s+\epsilon)=i(s) + \nu_s.
\]
\end{rem}

\end{document}